\documentclass[11pt]{article}
\usepackage{CJK}
\usepackage{amsmath, amsfonts, amssymb, amsthm, amscd, amsbsy}
\usepackage{mathrsfs}
\usepackage{fancyhdr}
\usepackage{geometry}
\usepackage{graphicx}
\usepackage{enumerate}
\usepackage{indentfirst}
\usepackage{color}
\usepackage{hyperref,amssymb,cite}
\usepackage{cite}
\usepackage[misc,geometry]{ifsym}
\providecommand{\U}[1]{\protect\rule{.1in}{.1in}}
\hypersetup{
colorlinks=true,
linkcolor=blue,
anchorcolor=blue,
citecolor=blue}
\theoremstyle{plain}
\newtheorem{theorem}{Theorem}[section]

\newtheorem{lemma}[theorem]{Lemma}

\numberwithin{equation}{section}
\newtheorem{theoremalph}{Theorem}

\newcommand{\HT}{\CJKfamily{hei}}

\newcommand{\R}{\mathbb{R}}
\newcommand{\N}{\mathbb{N}}

\begin{document}

\title{\HT { Existence of extremal functions and Wulff symmetry for  anisotropic Trudinger-Moser inequalities }}
\author{\small Kaiwen  Guo, \quad  \small  Yanjun Liu }
\date{ }
\maketitle
\noindent{\bf Abstract:} In this paper, we investigate the extremal functions  for  anisotropic Trudinger-Moser inequalities. Our method uses convex symmetrization, the continuity of the supremum function, together with the relation between the supremums of the subcritical  and the critical anisotropic Trudinger-Moser inequality, finally, we give some results of existence and symmetry about the extremal functions for several different types of anisotropic Trudinger-Moser inequalities. 
\section{Introduction and main results}\label{section1}
As is known to all, in the limiting case, the Sobolev embeddings are replaced by the Trudinger-Moser inequalities. Let $\Omega$ be a bounded domain in $\R^{N}$, $N\geq2$, Trudinger \cite{Trudinger-1967} (see also  Poho\v{z}aev \cite{Pohozaev-1965} and Yudovi\v{c} \cite{Yudovic-1961}) proved that $W^{1,N}_{0}(\Omega)$ is embedded in the Orlicz space $L_{\varphi_{\alpha}}(\Omega)$ which is determined by Young function $\varphi_{\alpha}(t)=e^{\alpha\lvert t\rvert^{\frac{N}{N-1}}}-1$ for some positive number $\alpha$. In particular, Moser \cite{Moser-1970/71} obtained the sharp constant  $\alpha_{N}=N^{\frac{N}{N-1}}\omega_{N}^{\frac{1}{N-1}}$ such that
\begin{equation}\label{eq1.1}
\sup_{u\in W_0^{1,N}(\Omega),\;\|\nabla u\|_{N}\leq1}\int_{\Omega}e^{\alpha_{N}|u|^{\frac{N}{N-1}}}\;\mathrm{d}x<+\infty,
\end{equation}
where $\omega_{N}$ is the volume of
the unit ball in $\mathbb{R}^{N}$.

The Trudinger-Moser inequality on domains with infinite volume are taken the
following forms.
 Let
$$
\phi_{N}(t)=e^{t}-\sum\limits_{j=0}^{N-2}\frac{t^{j}}{j!}
$$
and $\alpha\in(0,\alpha_{N})$, then
\begin{align}
    \sup\limits_{\|\nabla u\|_{N}\leq 1}\frac{1}{\|u\|_{N}^{N}}\int_{\R^{N}}\phi_{N}(\alpha\lvert u\rvert^{\frac{N}{N-1}})\;\mathrm{d}x<+\infty,\label{eq1.2}
\end{align}
the subcritical inequality \eqref{eq1.2} was established by Cao \cite{Cao-1992} in dimension two, do \'{O} \cite{do-1997} and Adachi-Tanaka \cite{Adachi-Tanaka-2000} in high dimension. For critical case,  by replacing the Dirichlet norm with the standard Sobolev
norm in $W^{1, N}(\R^{N})$, it holds
\begin{align}
    \sup\limits_{\|\nabla u\|_{N}^{N}+\|u\|_{N}^{N}\leq 1}\int_{\R^{N}}\phi_{N}(\alpha_N\lvert u\rvert^{\frac{N}{N-1}})\;\mathrm{d}x<+\infty,\label{eq1.3}
\end{align}
which was established by Ruf \cite{Ruf-2005}(for the case $N=2$) and Li-Ruf \cite{Li-Ruf-2008}(for the general case $N\geq2$). It is interesting to notice that the Trudinger-Moser type inequality can only be established for the subcritical case when  the term $\|\nabla u\|_{N}$ is  used in restriction of function class. Indeed, \eqref{eq1.2} has been proved in \cite{do-1997} and \cite{Adachi-Tanaka-2000} if $\alpha<\alpha_{N}$. Futhermore, their conclusions are actually sharp in the sense that the supremum is infinity when $\alpha\geq\alpha_{N}$. In order to achieve the critical case $\alpha=\alpha_{N}$, Ruf \cite{Ruf-2005} and then Li-Ruf \cite{Li-Ruf-2008} need to use the full form in $W^{1,N}(\R^{N})$, namely, $(\|u\|_{N}^{N}+\|\nabla u\|_{N}^{N})^{\frac{1}{N}}$. They also obtained that $\alpha_{N}$ is sharp without accident. One should note that all these results rely on the classic P\'{o}lya-Szeg\"{o} inequality and the Schwartz symmetrization argument, while the P\'{o}lya-Szeg\"{o} inequality  fails in other non-Euclidean settins such as the Heisenberg group, Riemannian manifolds  and  high order Sobolev spaces.  An alternative proof of \eqref{eq1.3} without using symmetrization has been given by Lam-Lu \cite{Lam-Lu-2013}. Different proofs of \eqref{eq1.2} and \eqref{eq1.3} have also been given without using symmetrization in settings such as on the Heisenberg group or high and fractional order Sobolev space where the symmetrization is not available (see the work in Lam-Lu \cite{Lam-Lu-2012} and Lam-Lu-Tang \cite{Lam-Lu-Tang-2014} which extend the earlier work by Cohn-Lu \cite{Cohn-Lu-2001} on finite domains). Such symmetrization-free method is also used to establish the Trudinger-Moser inequalities on the Riemannian manifolds by Li-Lu \cite{Li-Lu-2021} and on the Heisenberg group by Li-Lu-Zhu \cite{Li-Lu-Zhu-2018}. These represent some important progress in the subject of the study of Trudinger-Moser inequalities. Recently,  the inequality \eqref{eq1.3} has been essentially improved by Chen-Lu-Zhu \cite{Chen-Lu-Zhu-2021, Chen-Lu-Zhu-2021'} through replacing $\int_{\mathbb
{R}^{N}}(\lvert\nabla u\rvert^{N}+\lvert u\rvert^{N})\;\mathrm{d}x$ with $\int_{\mathbb
{R}^{N}}(\lvert\nabla u\rvert^{N}+V(x)\lvert u\rvert^{N})\;\mathrm{d}x$, where the potential $V(x)$ is non-negative and is assumed only having the lower bound at infinity.
 It can help to remove the positive lower bound assumption on potential when studying the ground state solution for Schr\"{o}dinger equation with the Trudinger-Moser growth.
 This improved Trudinger-Moser inequality and corresponding ground-state solution results of Schr\"{o}dinger equation has also been generalized to Heisenberg group in Chen-Lu-Zhu \cite{Chen-Lu-Zhu-2023'}
 and Adams case in Chen-Lu-Zhu \cite{Chen-Lu-Zhu-2023}. By Fourier-rearrangement-free argument, constructing concentration-compactness and suitable comparison principle for Trudinger-Moser type equation, Sun-Song-Zou\cite{Sun-Song-Zou} also obtianed  a sharp Trudinger-Moser type inequality involving $L^p$
norm  with degenerate potential.

Moreover, Ibrahim-Masmoudi-Nakanishi \cite{Ibrahim-Masmoudi-Nakanishi-2015}  obtained Trudinger-Moser inequalities with the exact growth in in dimension two and
Masmoudi--Sani \cite{Masmoudi-Sani-2015}  obtained the similar results in general dimension, the same type of exact growth inequality was proved even in hyperbolic spaces, see\cite{Lu-Tang-2016}. These  inequalities  play an important role in geometric analysis and partial differential equations, we refer  to \cite{Chang-Yang-2003,de- Figueiredo-Miyagaki-Ruf-1995, Adimurthi-Yang-2010, Lam-Lu-2014, do-O-Souza-de- Medeiros-Severo-2014, Masmoudi-Sani-2015} and references therein. In \cite{Lam-Lu-Zhang-2017},   Lam-Lu-Zhang  provide a precise relationship between  subcritical and
critical  Trudinger-Moser inequalities. The similar result in Lorentz-Sobolev norms was also proved
by Lu-Tang \cite{Lu-Tang-2016'}.
Trudinger-Moser inequality for first order derivatives was extended to high order derivatives by D. Adams \cite{Adams-1988} for
bounded domains when  dimension $N\geq2$. Ruf-Sani \cite{ Ruf-Sani-2013} studied the Adams type inequality with
higher derivatives of even orders for unbounded domains in $\mathbb{R}^{N}$. Adams inequalities with the exact growth were established by  Masmoudi-Sani\cite{Masmoudi-Sani-2014} in dimension four and then established in general dimension by Lu-Tang-Zhu\cite{Lu-Tang-Zhu-2015}(see \cite{Masmoudi-Sani-2018} for higher order case), some existence results of extremal functions can refer to \cite{Chen-Lu-Zhu-2020}, moreover, in
\cite{Zhang-Yang-Cheng-Zhou-2025}, the authors extended the results of \cite{Masmoudi-Sani-2014} to the anisotropic case.

One important problem on Trudinger-Moser inequality is whether extremal functions exist or not. The existence of extremal functions for inequality \eqref{eq1.1} was firstly obtained by Carleson-Chang \cite{Carleson-Chang-1986} when $\Omega$ is the unit ball, then obtained by Struwe \cite{Struwe-1988} when $\Omega$ is close to the unit ball in the sense of measure, finally obtained by Flucher \cite{Flucher-1992} and Lin \cite{Lin-1996} when $\Omega$ is a smooth bounded domain.
Recently,
 the author of \cite{Chen-Jiang-Lu-Zhu-2025} founds the effect of sharp $L^{p}$ perturbation
 on the existence and non-existence for extremals of Trudinger-Moser inequality in two-dimensional bounded domain. And in fact, effect of polynomial perturbation on the existence and
 non-existence of extremals for critical Moser-Trudinger inequality in $\mathbb{R}^{2}$
 and Adams inequality in $\mathbb{R}^{4}$ or $\mathbb{R}^{2m}$ for $m>2$ was earlier systematically studied in \cite{Chen-Lu-Zhu-2020} and \cite{Chen-Lu-Zhu-2022}. The existence of extremal functions for inequality \eqref{eq1.3} was obtained by Ruf \cite{Ruf-2005}, Li-Ruf \cite{Li-Ruf-2008} and Ishiwata \cite{Ishiwata-2011}. On compact Riemannian manifolds, using the blow-up
analysis of the Euler--Lagrange equations, Li established the existence
of extremal functions for the Trudinger--Moser inequalities, see \cite{Li-2001,Li-2005,Li-2006}.  Recently based on the work by Malchiodi-Martinazzi \cite{Malchiodi-Martinazzi-2014}, Mancini-Martinazzi \cite{Mancini-Martinazzi-2017} reproved the Carleson-Chang's result by using a new method based on the Dirichlet energy, also allowing
for perturbations of the functional. Besides, the author of \cite{Chen-Lu-Xue-Zhu-2022} established the uniqueness for local Moser-Trudinger
 equation and it is an important step towards solving the uniqueness of extremal function of Trudinger-Moser inequality in disk.

There are also several other related expansions of the Trudinger-Moser inequality.
When $\Omega$ contains the origin, Adimurthi-Sandeep \cite{Adimurthi-Sandeep-2007} generalized inequality \eqref{eq1.1} to a singular version, namely, for any $0<\beta<N$,
\begin{align}
    \sup\limits_{u\in W_{0}^{1,N}(\Omega),\;\|\nabla u\|_{N}\leq 1}
    \int_{\Omega}\frac{e^{\alpha_{N}(1-\frac{\beta}{N})\lvert u\rvert^{\frac{N}{N-1}}}}{\lvert x\rvert^{\beta}}\;\mathrm{d}x<+\infty,\notag
\end{align}
  for extremal functions of singular version, Csat\'{o}-Roy \cite{Csato-Roy-2015} proved that extremal functions exist in bounded domain of  two dimension, and the rusult of general dimension by  Csat\'{o}-Roy-Nguyen\cite{Csato-Nguyen-Roy-2021}.
Adimurthi-Yang \cite{Adimurthi-Yang-2010} generalized the inequality to $\R^{N}$ as follows:
\begin{align}\label{eq1.4}
    \sup\limits_{u\in W^{1,N}(\R^{N}),\;\int_{\R^{N}}(\lvert\nabla u\rvert^{N}+\lvert u\rvert^{N})\;\mathrm{d}x\leq 1}
    \int_{\R^{N}}\frac{\Phi_{N}(\alpha_{N}(1-\frac{\beta}{N})\lvert u\rvert^{\frac{N}{N-1}})}{\lvert x\rvert^{\beta}}\;\mathrm{d}x<+\infty,
\end{align}
and it was proven in Li-Yang \cite{Li-Yang-2018} that the supremum can be attained. Lam-Lu-Zhang \cite{Lam-Lu-Zhang-2019'} also established the existence
	and nonexistence of the maximizers for singular Trudinger-Moser
	inequalities in different ranges of the parameters.
Wang-Xia \cite{Wang-Xia-2012} investigated Trudinger-Moser inequality involving the anisotropic Dirichlet norm $\left(\int_{\Omega}F(\nabla u)^{N}\;\mathrm{d}x\right)^{\frac{1}{N}}$, precisely,
\begin{align}
    \sup_{u\in W_{0}^{1,N}(\Omega),\;\|F(\nabla u)\|_{N}\leq 1}
    \int_{\Omega}e^{\lambda_{N}\lvert u\rvert^{\frac{N}{N-1}}}\;\mathrm{d}x<+\infty,\notag
\end{align}
where $\lambda_{N}=N^{\frac{N}{N-1}}\kappa_{N}^{\frac{1}{N-1}}$ and $\kappa_{N}=|\{x \in \mathbb{R}^{N} : F^{o}(x)\leq 1\}|$ is the volume of a unit Wulff ball in $\mathbb{R}^N$, where $F:\mathbb{R}^{N}\rightarrow[0,+\infty)$ is a convex function of class $C^{2}(\mathbb{R}^{N}\setminus\{0\})$, which is even and positively homogeneous of degree 1, $F_{\xi_{i}} = \frac{\partial F}{\partial\xi_{i}}$ and its polar $F^{o}(x)$ represents a Finsler metric on $\mathbb{R}^{N}$. Subsequently, Zhou-Zhou \cite{Zhou-Zhou-2019}  generalized the inequality to $\R^{N}$ as follows:
\begin{align}\label{eq1.5}
    \sup_{u\in W^{1,N}(\R^{N}),\;\int_{\R^{N}}(F(\nabla u)^{N}+\lvert u\rvert^{N})\;\mathrm{d}x\leq1}\int_{\R^{N}}\Phi_{N}(\lambda_{N}\lvert u\rvert^{\frac{N}{N-1}})\;\mathrm{d}x<+\infty,
\end{align}
and they  obtained the existence of extremal functions.  In \cite{Liu}, the second author established the anisotropic Trudinger-Moser inequality associated with exact growth in $\R^{N}$,  moreover, the author also calculated the supremum of the inequality, and the existence and nonexistence of extremal functions  are also obtained in certain cases.
It is worth mentioning that in Lu-Shen-Xue-Zhu \cite{Lu-Shen-Xue-Zhu-2024}, they established the weighted anisotropic isoperimetric inequalities, which can be used to study the anisotropic Trudinger-Moser inequality.

In this paper, we will always assume that
\begin{align}\label{eq1.6}
    N\geq 2,\;\;0\leq\beta<N,\;\;0\leq\lambda<\lambda_{N}:=N^{\frac{N}{N-1}}\kappa_{N}^{\frac{1}{N-1}},\;\;q\geq 1,
\end{align}
and consider the function
$$
\Phi_{N,q,\beta}(t):=\left\{\begin{array}{cc}
\sum\limits_{j\in\N,\;j>\frac{q(N-1)}{N}(1-\frac{\beta}{N})}\frac{t^{j}}{j!} & {\mathrm{if}}\;\beta>0 \\
\sum\limits_{j\in\N,\;j\geq\frac{q(N-1)}{N}}\frac{t^{j}}{j!} & {\mathrm{if}}\;\beta=0
\end{array}\right..
$$
In the recent paper \cite{Guo-Liu-2024}, we have established the following results, where the supremums are for the functions in $D^{N,q}(\R^{N})$, the completion of $C_{0}^{\infty}(\R^{N})$ under the norm $\|\nabla u\|_{N}+\|u\|_{q}$.

\begin{theoremalph}(Sharp subcritical anisotropic Trudinger-Moser inequality).\label{tha}
Let $p>q(1-\frac{\beta}{N})$ if $\beta\neq0$ and $p\geq q$ if $\beta=0$. Then
\begin{align}
    \mathrm{ATMSC}(N,p,q,\lambda,\beta)&:=\sup\limits_{\|F(\nabla u)\|_{N}\leq 1}\frac{1}{\|u\|_{q}^{q(1-\frac{\beta}{N})}}\int\limits_{\R^{N}}\frac{\exp(\lambda(1-\frac{\beta}{N})\lvert u\rvert^{\frac{N}{N-1}})\lvert u\rvert^{p}}{F^{o}(x)^{\beta}}\;\mathrm{d}x<+\infty,\notag\\
    \mathrm{ATMSC}(N,q,\lambda,\beta)&:=\sup\limits_{\|F(\nabla u)\|_{N}\leq 1}\frac{1}{\|u\|_{q}^{q(1-\frac{\beta}{N})}}\int\limits_{\R^{N}}\frac{\Phi_{N,q,\beta}(\lambda(1-\frac{\beta}{N})\lvert u\rvert^{\frac{N}{N-1}})}{F^{o}(x)^{\beta}}\;\mathrm{d}x<+\infty,\notag
\end{align}
\end{theoremalph}

\begin{theoremalph}(Sharp critical anisotropic Trudinger-Moser inequality).\label{thb}
Let $a>0$ and $b>0$. Then
\begin{align}
    \mathrm{ATMC}_{a,b}(N,q,\beta):=\sup\limits_{\|F(\nabla u)\|_{N}^{a}+\|u\|_{q}^{b}\leq 1}\int\limits_{\R^{N}}\frac{\Phi_{N,q,\beta}(\lambda_{N}(1-\frac{\beta}{N})\lvert u\rvert^{\frac{N}{N-1}})}{F^{o}(x)^{\beta}}\;\mathrm{d}x<\infty\Leftrightarrow b\leq N.\notag
\end{align}
The constant $\lambda_{N}$ is sharp.
\end{theoremalph}
When $F=\lvert\cdot\rvert$, the above results are the isotropic Trudinger-Moser inequality in Lam-Lu-Zhang \cite{Lam-Lu-Zhang-2019}.  Motivated by the results in \cite{Lam-Lu-Zhang-2017} and \cite{Lam-Lu-Zhang-2019'}, based on the relationship between  subcritical and
critical  Trudinger-Moser inequalities, they  obtained the existence and symmetry of extremal functions under isotropic case, these results provide us with valuable insights, here we consider  the anosotropic case, our first aim  is to study the maximizers for the subcritical anisotropic Trudinger-Moser inequality, and $u^{\star}$ as the convex symmetrization of $u$ with respect to $F$ is defined in Section \ref{section2}.

\begin{theorem}\label{th1.1}
    Let $p>q(1-\frac{\beta}{N})$ if $\beta\neq0$, $p\geq q$ if $\beta=0$ and $q>1$. Then the supremums of the subcritical anisotropic Trudinger-Moser inequalities $\mathrm{ATMSC}(N,p,q,\lambda,\beta)$ and $\mathrm{ATMSC}(N,q,\lambda,\beta)$ can be attained by some nonnegative, Wulff symmetric $u_{0}$ satisfying $\|F(\nabla u_{0})\|_{N}=\|u_{0}\|_{q}=1$. Moreover, all the maximizers $u$ satisfy $\|F(\nabla u)\|_{N}=\|F(\nabla u^{\star})\|_{N}=1$, and $u$ can be adjusted to be nonnegative, Wulff symmetric function $v$ with $\|F(\nabla v)\|_{N}=\|F(\nabla v^{\star})\|_{N}=\|v\|_{q}=1$.
\end{theorem}

In the singular case, we will show that the absolute values of the maximizers are Wulff symmetric.

\begin{theorem}\label{th1.2}
    Let $0<\beta<N$, $p>q(1-\frac{\beta}{N})$ and $q>1$. Suppose $u$ is a maximizer of the supremum $\mathrm{ATMSC}(N,p,q,\lambda,\beta)$ or $\mathrm{ATMSC}(N,q,\lambda,\beta)$. Then $\lvert u\rvert=\lvert u\rvert^{\star}$, namely, $\lvert u\rvert$ is Wulff symmetric.
\end{theorem}

Our next aim is to study the existence and nonexistence of the maximizers for the critical anisotropic Trudinger-Moser inequality in the subcritical case. The following identity has been established in \cite{Guo-Liu-2024}:
$$
{\mathrm{ATMC}}_{a,b}(N,q,\beta)=\sup\limits_{\lambda\in(0,\lambda_{N})}
\left(\frac{1-(\frac{\lambda}{\lambda_{N}})^{a\frac{N-1}{N}}}{(\frac{\lambda}{\lambda_{N}})^{b\frac{N-1}{N}}}\right)^{\frac{q}{b}(1-\frac{\beta}{N})}{\mathrm{ATMSC}}(N,q,\lambda,\beta).
$$
Now, define
$$
\Lambda_{a,b}(N,q,\lambda,\beta):=\sup\limits_{\|F(\nabla u)\|_{N}^{a}+\|u\|_{q}^{b}\leq 1}\int\limits_{\R^{N}}\frac{\Phi_{N,q,\beta}(\lambda(1-\frac{\beta}{N})\lvert u\rvert^{\frac{N}{N-1}})}{F^{o}(x)^{\beta}}\;\mathrm{d}x.
$$
Using the same approach, we can prove
\begin{align}\label{AA}
    \Lambda_{a,b}(N,q,\lambda,\beta)=\sup\limits_{t\in(0,\lambda)}\left(\frac{1-(\frac{t}{\lambda})^{a\frac{N-1}{N}}}{(\frac{t}{\lambda})^{b\frac{N-1}{N}}}\right)^{\frac{q}{b}(1-\frac{\beta}{N})}{\mathrm{ATMSC}}(N,q,t,\beta).
\end{align}
Moreover, we have
\begin{theorem}\label{th1.3}
    Let $a>0$, $b>0$, $q>1$ and $0<\lambda<\lambda_{N}$.

    \textup{(i)} If $\beta>0$, then the supremum $\Lambda_{a,b}(N,q,\lambda,\beta)$ can be attained and the absolute values of the maximizers are Wulff symmetric.

    \textup{(ii)} If $\beta=0$ and $\frac{q(N-1)}{N}\not\in\N$, then the supremum $\Lambda_{a,b}(N,q,\lambda, 0)$ can be attained.

    \textup{(iii)} If $\beta=0$ and $\frac{q(N-1)}{N}\in\N$, then the supremum $\Lambda_{a,b}(N,q,\lambda, 0)$ is not attained only if $$\Lambda_{a,b}(N,q,\lambda,0)\leq\frac{\lambda^{\frac{q(N-1)}{N}}}{\left(\frac{q(N-1)}{N}\right)!}.$$
\end{theorem}

This paper is organized as follows. In Section \ref{section2}, we will provide some preliminary information and prove some lemmas which will be used to prove our theorems. In Section \ref{section3}, we provide the proofs about the maximizers for the subcritical anisotropic Trudinger-Moser inequality. Namely, we prove Theorem \ref{th1.1} and Theorem \ref{th1.2}. In Section \ref{section4}, we provide the proof of Theorem \ref{th1.3}.
\section{Preliminaries}\label{section2}
In this section, we provide some preliminary information that will be used in our proofs.

Let $F:\R^{N}\rightarrow[0,+\infty)$ be a convex function of class $C^{2}(\R^{N}\backslash\{0\})$, which is even and positively homogeneous of degree 1, then there holds
$$
F(tx)=\lvert t\rvert F(x)\;\;\mathrm{for}\;\mathrm{any}\;t\in\R,\;x\in\R^{N}.
$$

We further assume $F(x)>0$ for any $x\neq 0$ and $Hess(F^{2})$ is positive definite in $\R^{N}\backslash\{0\}$, which leading $Hess(F^{N})$ is positive definite in $\R^{N}\backslash\{0\}$ by Xie and Gong \cite{Xie-Gong-2016}. There are two constants $0<a\leq b<\infty$ such that $a\lvert x\rvert\leq F(x)\leq b\lvert x\rvert$ for any $x\in\R^{N}$ and a typical example of $F$ is $F(x):=(\sum\limits_{i=1}^{N}\lvert x_{i}\rvert^{q})^{\frac{1}{q}}$ for $q\in(1,+\infty)$.

Considering the minimization problem
$$
\min_{u}\int\limits_{\R^{N}}F(\nabla u)^{N}\;\mathrm{d}x,
$$
its Euler-Lagrange equation contains an operator of the form
$$
Q_{N}(u):=\sum_{i=1}^{N}\frac{\partial}{\partial x_{i}}(F(\nabla u)^{N-1}F_{x_{i}}(\nabla u)),
$$
which is called as $N$-anisotropic Laplacian or $N$-Finsler Laplacian.

Let $F^{o}$ be the support function of $K:=\{x\in\R^{N}:F(x)\leq 1\}$, which is defined by
$$
F^{o}(x):=\sup_{\xi\in K}\langle x,\xi\rangle,
$$
then $F^{o}:\R^{N}\rightarrow[0,+\infty)$ is also a convex, positively homogeneous function of class $C^{2}(\R^{N}\backslash\{0\})$.

From \cite{Alvino-Ferone-Trombetti-Lions-1997}, $F^{o}$ is dual to $F$ in the sense that
$$
F^{o}(x)=\sup_{\xi\neq 0}\frac{\langle x,\xi\rangle}{F(\xi)}, \;F(x)=\sup_{\xi\neq 0}\frac{\langle x,\xi\rangle}{F^{o}(\xi)}.
$$

Consider a map $\Phi:S^{N-1}\rightarrow\R^{N}$ satisfying $\Phi(\xi)=\nabla F(\xi)$. Its image $\Phi(S^{N-1})$ is a smooth, convex hypersurface in $\R^{N}$, which is known as the Wulff shape or equilibrium crystal shape of $F$. As a result, $\Phi(S^{N-1})=\{x\in\R^{N}|F^{o}(x)=1\}$ (see \cite{Wang-Xia-2011}, Proposition 1). Denote $W_{r}^{x_{0}}=\{x\in\R^{N}:F^{o}(x-x_{0})\leq r\}$, then we call $W_{r}$ as a Wulff ball with radius $r$ and center at the origin and denote $\kappa_{N}$ as the Lebesgue measure of $W_{1}$.

Accordingly, we provide some simple properties of $F$, as a direct consequence of assumption on $F$, also found in \cite{Bellettini-Paolini-1996,Ferone-Kawohl-2009,Wang-Xia-2011,Wang-Xia-2012}.
\begin{lemma}\label{le2.1}
There are\\
\textup{(i)} $\lvert F(x)-F(y)\rvert\leq F(x+y)\leq F(x)+F(y)$\textup{;}\\
\textup{(ii)} $\frac{1}{C}\leq\lvert\nabla F(x)\rvert\leq C$ and $\frac{1}{C}\leq \lvert\nabla F^{o}(x)\rvert\leq C$ for some $C>0$ and any $x\neq0$\textup{;}\\
\textup{(iii)} $\langle x,\nabla F(x)\rangle=F(x), \langle x, \nabla F^{o}(x)\rangle = F^{o}(x)$ for any $x\neq 0$\textup{;}\\
\textup{(iv)} $\nabla F(tx)=sgn(t)\nabla F(x)$ for any $x\neq 0$ and $t\neq 0$\textup{;}\\
\textup{(v)} $\int\limits_{\partial W_r}\frac{1}{|\nabla F^{o}(x)|}\;\mathrm{d}\sigma= N\kappa_Nr^{N-1}$\textup{;}\\
\textup{(vi)} $F(\nabla F^{o}(x))=1, F^{o}(\nabla F(x))=1$   for any $x\neq 0$\textup{;}\\
\textup{(vii)} $F(x)\nabla F^{o}(\nabla F(x))=x, F^{o}(x)\nabla F(\nabla F^{o}(x))=x$ for any $x\neq 0$.
\end{lemma}

The convex symmetrization defined in \cite{Alvino-Ferone-Trombetti-Lions-1997} will be used later, which generalizes the Schwarz symmetrization in \cite{Talenti-1976}. For any measurable set $\Omega\subset\R^{N}$, let us consider a measurable function $u$ on $\Omega$, one-dimensional decreasing rearrangement of $u$ is
$$
u^{\sharp}(t)=\sup\{s\geq 0:\lvert\{x\in\Omega:u(x)\geq s\}\rvert >t\},\;t\in\R.
$$
The convex symmetrization of $u$ with respect to $F$ is defined by
$$
u^{\star}(x)=u^{\sharp}\left(\kappa_{N}F^{o}(x)^{N}\right),\;x\in\Omega^{\star},
$$
where $\kappa_{N}F^{o}(x)^{N}$ is the Lebesgue measure of $W_{F^{o}(x)}$ and $\Omega^{\star}$ is a Wulff ball centered at the origin having the same Lebesgue measure as $\Omega$. Recalling the Schwarz symmetrization of $u$ is defined by
$$
u^{\diamond}(x)=u^{\sharp}(\omega_{N}\lvert x\rvert^{N}),\;\;\;x\in\Omega^{\diamond},
$$
where $\omega_{N}\lvert x\rvert^{N}$ is the Lebesgue measure of $B_{\lvert x\rvert}$ and $\Omega^{\diamond}$ is a Euclidean ball centered at the origin having the same Lebesgue measure as $\Omega$. It is easy to see that the Schwarz symmetrization is the convex symmetrization when $F=\lvert\cdot\rvert$.

In addition, the convex symmetrization has the following characteristic.
\begin{lemma}\label{le2.2}(P$\acute{o}$lya-Szeg$\ddot{o}$ Inequality)
Let $u\in D^{N,q}(\R^{N})$. Then $u^{\star}\in D^{N,q}(\R^{N})$ and
\begin{align}
    &\int\limits_{\R^{N}}F(\nabla u^{\star})^{N}\;\mathrm{d}x   \leq\int\limits_{\R^{N}}F(\nabla u)^{N}\;\mathrm{d}x.\notag
\end{align}
\end{lemma}
\begin{lemma}\label{le2.3}(Hardy-Littlewood Inequality)
Let $f$, $g$ be nonnegative functions on $\R^{N}$, vanishing at infinity. Then
$$
\int\limits_{\R^{N}}f(x)g(x)\;\mathrm{d}x\leq \int\limits_{\R^{N}}f^{\star}(x)g^{\star}(x)\;\mathrm{d}x.
$$
Moreover, if $f$ is strictly decreasing and Wulff symmetric, then there is equality if and only if $g=g^{\star}$.
\end{lemma}

Next, we prove some lemmas that will be used.

\begin{lemma}\label{le2.4}
Let $\Omega\subset\R^{N}$, $\lvert\Omega\rvert<+\infty$. Assume that
$$
f_{n}\rightarrow f\;a.e.\;in\;\Omega
$$
and there exists $q>1$ such that $f_{n}$ is uniformly bounded in $L^{q}(\Omega)$, $f\in L^{q}(\Omega)$. Then
$$
f_{n}\rightarrow f\;in\;L^{1}(\Omega).
$$
\end{lemma}
\begin{proof}
For any $\epsilon>0$, by Egorov's theorem, there exists a measurable set $E\subset\Omega$ such that
$$
f_{n}\rightarrow f\;\mathrm{uniformly}\;\mathrm{in}\;E\;\mathrm{and}\;\lvert\Omega\backslash E\rvert<\epsilon.
$$
Thus,
$$
\int\limits_{E}\lvert f_{n}-f\rvert\;\mathrm{d}x\rightarrow0.
$$
On the other hand, by Holder's inequality,
$$
\int\limits_{\Omega\backslash E}\lvert f_{n}-f\rvert\;\mathrm{d}x\leq\bigg(\int\limits_{ \Omega\backslash E}\lvert f_{n}-f\rvert^{q}\;\mathrm{d}x\bigg)^{1/q}\bigg(\int\limits_{ \Omega\backslash E}1^{q^{\ast}}\;\mathrm{d}x\bigg)^{1/q^{\ast}}\leq C\epsilon^{1/q^{\ast}},
$$
where $\frac{1}{q}+\frac{1}{q^{\ast}}=1$. Therefore, $f_{n}\rightarrow f$ in $L^{1}(\Omega)$.
\end{proof}

\begin{lemma}\label{le2.5}
Let $p>q(1-\frac{\beta}{N})$ if $\beta\neq0$ and $p\geq q$ if $\beta=0$. Then
\begin{align}
    \mathrm{ATMSC}(N,p,q,\lambda,\beta)&=\sup\limits_{\|F(\nabla u)\|_{N}=\|u\|_{q}=1}\int\limits_{\R^{N}}\frac{\exp(\lambda(1-\frac{\beta}{N})\lvert u\rvert^{\frac{N}{N-1}})\lvert u\rvert^{p}}{F^{o}(x)^{\beta}}\;\mathrm{d}x,\notag\\
    \mathrm{ATMSC}(N,q,\lambda,\beta)&=\sup\limits_{\|F(\nabla u)\|_{N}=\|u\|_{q}=1}\int\limits_{\R^{N}}\frac{\Phi_{N,q,\beta}(\lambda(1-\frac{\beta}{N})\lvert u\rvert^{\frac{N}{N-1}})}{F^{o}(x)^{\beta}}\;\mathrm{d}x.\notag
\end{align}
\end{lemma}
\begin{proof}
We just prove the first identity. For any $u\in D^{N,q}(\R^{N})$ with $\|F(\nabla u)\|_{N}\leq1$, let
$$
v(x)=\frac{u(tx)}{\|F(\nabla u)\|_{N}},\;t=\left(\frac{\|u\|_{q}}{\|F(\nabla u)\|_{N}}\right)^{q/N},
$$
where we can easily exclude the case $\|F(\nabla u)\|_{N}=0$. There are
$$
\|F(\nabla v)\|_{N}=1,\;\|v\|_{q}=1,
$$
and
\begin{align}
    &\int\limits_{\R^{N}}\frac{\exp(\lambda(1-\frac{\beta}{N})\lvert v\rvert^{\frac{N}{N-1}})\lvert v\rvert^{p}}{F^{o}(x)^{\beta}}\;\mathrm{d}x\notag\\
    =&\frac{1}{\|F(\nabla u)\|_{N}^{p}}\int\limits_{\R^{N}}\exp\left(\lambda(1-\frac{\beta}{N})(\frac{\lvert u(tx)\rvert}{\|F(\nabla u)\|_{N}})^{\frac{N}{N-1}}\right)\frac{\lvert u(tx)\rvert^{p}}{F^{o}(x)^{\beta}}\;\mathrm{d}x\notag\\
    =&\frac{1}{\|F(\nabla u)\|_{N}^{p}}\frac{1}{t^{N-\beta}}\int\limits_{\R^{N}}\exp\left(\lambda(1-\frac{\beta}{N})(\frac{\lvert u(x)\rvert}{\|F(\nabla u)\|_{N}})^{\frac{N}{N-1}}\right)\frac{\lvert u(x)\rvert^{p}}{F^{o}(x)^{\beta}}\;\mathrm{d}x\notag\\
    =&\frac{1}{\|F(\nabla u)\|_{N}^{p-q(1-\frac{\beta}{N})}}\frac{1}{\|u\|_{q}^{q(1-\frac{\beta}{N})}}\int\limits_{\R^{N}}\exp\left(\lambda(1-\frac{\beta}{N})(\frac{\lvert u(x)\rvert}{\|F(\nabla u)\|_{N}})^{\frac{N}{N-1}}\right)\frac{\lvert u(x)\rvert^{p}}{F^{o}(x)^{\beta}}\;\mathrm{d}x\notag\\
    \geq&\frac{1}{\|u\|_{q}^{q(1-\frac{\beta}{N})}}\int\limits_{\R^{N}}\frac{\exp(\lambda(1-\frac{\beta}{N})\lvert u\rvert^{\frac{N}{N-1}})\lvert u\rvert^{p}}{F^{o}(x)^{\beta}}\;\mathrm{d}x.\notag
\end{align}
\end{proof}

\begin{lemma}\label{le2.6}
Let $a>0$ and $b>0$. Then
$$
\Lambda_{a,b}(N,q,\lambda,\beta)=\sup\limits_{\|F(\nabla u)\|_{N}^{a}+\|u\|_{q}^{b}=1}\int\limits_{\R^{N}}\frac{\Phi_{N,q,\beta}(\lambda(1-\frac{\beta}{N})\lvert u\rvert^{\frac{N}{N-1}})}{F^{o}(x)^{\beta}}\;\mathrm{d}x.
$$
\end{lemma}
\begin{proof}
For any $u\in D^{N,q}(\R^{N})$ with $\|F(\nabla u)\|_{N}^{a}+\|u\|_{q}^{b}\leq1$, we can choose some $c\geq1$ such that $c^{a}\|F(\nabla u)\|_{N}^{a}+c^{b}\|u\|_{q}^{b}=1$. There is
$$
\int\limits_{\R^{N}}\frac{\Phi_{N,q,\beta}(\lambda(1-\frac{\beta}{N})\lvert cu\rvert^{\frac{N}{N-1}})}{F^{o}(x)^{\beta}}\;\mathrm{d}x\geq\int\limits_{\R^{N}}\frac{\Phi_{N,q,\beta}(\lambda(1-\frac{\beta}{N})\lvert u\rvert^{\frac{N}{N-1}})}{F^{o}(x)^{\beta}}\;\mathrm{d}x.
$$
\end{proof}

\section{Maximizers for the subcritical anisotropic Trudinger-Moser inequality}\label{section3}

In this section, we will prove the existence of maximizers for the subcritical anisotropic Trudinger-Moser supremums $\mathrm{ATMSC}(N,p,q,\lambda,\beta)$ and $\mathrm{ATMSC}(N,q,\lambda,\beta)$, and we will show that the absolute values of the maximizers are Wulff symmetric in the singular case.

\subsection{Proof of Theorem \ref{th1.1}}
\begin{proof}
Let $\{\omega_{n}\}$ be a maximizing sequence of
$$
\mathrm{ATMSC}(N,p,q,\lambda,\beta)=\sup\limits_{\|F(\nabla u)\|_{N}\leq 1}\frac{1}{\|u\|_{q}^{q(1-\frac{\beta}{N})}}\int\limits_{\R^{N}}\frac{\exp(\lambda(1-\frac{\beta}{N})\lvert u\rvert^{\frac{N}{N-1}})\lvert u\rvert^{p}}{F^{o}(x)^{\beta}}\;\mathrm{d}x.
$$
Then by Lemma \ref{le2.2}, Lemma \ref{le2.3} and Lemma \ref{le2.5}, we may assume that $\omega_{n}$ is nonnegative, Wulff symmetric and
$$
\int\limits_{\R^{N}}\frac{\exp(\lambda(1-\frac{\beta}{N})\lvert \omega_{n}\rvert^{\frac{N}{N-1}})\lvert \omega_{n}\rvert^{p}}{F^{o}(x)^{\beta}}\;\mathrm{d}x\rightarrow\mathrm{ATMSC}(N,p,q,\lambda,\beta)
$$
with
$$
\|F(\nabla\omega_{n})\|_{N}=\|\omega_{n}\|_{q}=1.
$$
For $q>1$, up to a subsequence, there exists some $\omega\in D^{N,q}(\R^{N})$ such that
\begin{align}
    &\omega_{n}\rightharpoonup \omega\;\;\mathrm{weakly}\;\mathrm{in}\;D^{N,q}(\R^{N})\textup{;}\notag\\
    &\omega_{n}\rightarrow \omega\;\;\mathrm{strongly}\;\mathrm{in}\;L^{N}(\R^{N})\;\mathrm{and}\;L_{\mathrm{loc}}^{i}(\R^{N})\;\mathrm{for}\;i\in(0,+\infty)\textup{;}\notag\\
    &\omega_{n}\rightarrow \omega\;\;\mathrm{a.e.}\;\mathrm{in}\;\R^{N}.\notag
\end{align}
Obviously, $\omega$ is nonnegative, Wulff symmetric and
$$
\|F(\nabla\omega)\|_{N}\leq1,\;\|\omega\|_{q}\leq1.
$$

{\bfseries Case 1:} $\beta=0$, $p>q$.

For any $x\not=0$,
\begin{align}
    \|\omega_{n}\|_{q}^{q}\geq\int_{0}^{F^{o}(x)}\omega_{n}(r)^{q}\int_{\partial W_{r}}\frac{1}{\lvert\nabla F^{o}(x)\rvert}\;\mathrm{d}\sigma\mathrm{d}r\geq\kappa_{N}F^{o}(x)^{N}\omega_{n}(F^{o}(x))^{q}.\notag
\end{align}
Thus, for any $\epsilon>0$, there exists $R>0$ sufficiently large such that $\lvert\omega_{n}(x)\rvert\leq\epsilon$ when $F^{o}(x)\geq R$. Then by Theorem \ref{tha},
\begin{align}
    \int\limits_{\R^{N}\backslash W_{R}}\exp(\lambda\lvert \omega_{n}\rvert^{\frac{N}{N-1}})\lvert \omega_{n}\rvert^{p}\;\mathrm{d}x&\leq\epsilon^{p-q}\int\limits_{\R^{N}\backslash W_{R}}\exp(\lambda\lvert \omega_{n}\rvert^{\frac{N}{N-1}})\lvert \omega_{n}\rvert^{q}\;\mathrm{d}x\notag\\
    &\leq C(N,q,\lambda)\epsilon^{p-q}.\notag
\end{align}
Next, notice that
$$
\exp(\lambda\lvert \omega_{n}\rvert^{\frac{N}{N-1}})\lvert \omega_{n}\rvert^{p}\rightarrow\exp(\lambda\lvert \omega\rvert^{\frac{N}{N-1}})\lvert \omega\rvert^{p}\;\mathrm{a.e.}\;\mathrm{in}\;W_{R}
$$
and by Theorem \ref{tha}, there exists $\theta>0$ sufficiently small such that $\exp(\lambda\lvert \omega_{n}\rvert^{\frac{N}{N-1}})\lvert \omega_{n}\rvert^{p}$ is uniformly bounded in $L^{1+\theta}(W_{R})$. Thus, by Lemma \ref{le2.4},
$$
\int\limits_{W_{R}}\exp(\lambda\lvert \omega_{n}\rvert^{\frac{N}{N-1}})\lvert \omega_{n}\rvert^{p}\;\mathrm{d}x\rightarrow\int\limits_{W_{R}}\exp(\lambda\lvert \omega\rvert^{\frac{N}{N-1}})\lvert \omega\rvert^{p}\;\mathrm{d}x.
$$
Hence,
\begin{align}
    \mathrm{ATMSC}(N,p,q,\lambda,\beta)\leftarrow&\int\limits_{\R^{N}}\exp(\lambda\lvert \omega_{n}\rvert^{\frac{N}{N-1}})\lvert \omega_{n}\rvert^{p}\;\mathrm{d}x\notag\\
    \leq&\int\limits_{\R^{N}}\exp(\lambda\lvert \omega\rvert^{\frac{N}{N-1}})\lvert \omega\rvert^{p}\;\mathrm{d}x+C(N,q,\lambda)\epsilon^{p-q}.\notag
\end{align}
Since $\epsilon>0$ is arbitrary, we obtain
$$
\mathrm{ATMSC}(N,p,q,\lambda,\beta)\leq\int\limits_{\R^{N}}\exp(\lambda\lvert \omega\rvert^{\frac{N}{N-1}})\lvert \omega\rvert^{p}\;\mathrm{d}x.
$$
Obviously, $\omega\not=0$ and
\begin{align}
    \mathrm{ATMSC}(N,p,q,\lambda,\beta)&\leq\int\limits_{\R^{N}}\exp(\lambda\lvert \omega\rvert^{\frac{N}{N-1}})\lvert \omega\rvert^{p}\;\mathrm{d}x\notag\\
    &\leq\frac{1}{\|\omega\|_{q}^{q}}\int\limits_{\R^{N}}\exp(\lambda\lvert \omega\rvert^{\frac{N}{N-1}})\lvert \omega\rvert^{p}\;\mathrm{d}x.\notag
\end{align}
Therefore, $\|\omega\|_{q}=1$ and
$$
\mathrm{ATMSC}(N,p,q,\lambda,\beta)=\int\limits_{\R^{N}}\exp(\lambda\lvert \omega\rvert^{\frac{N}{N-1}})\lvert \omega\rvert^{p}\;\mathrm{d}x.
$$
Moreover, $\omega$ is a maximizer of $\mathrm{ATMSC}(N,p,q,\lambda,\beta)$ with $\|F(\nabla\omega)\|_{N}\leq1$.

{\bfseries Case 2:} $\beta=0$, $p=q$.

Similarly as in the first case, we could show that for any $\epsilon>0$, there exists $R>0$ sufficiently large such that $\lvert\omega_{n}(x)\rvert\leq\epsilon$ when $F^{o}(x)\geq R$ and
$$
\int\limits_{W_{R}}\left[\exp(\lambda\lvert \omega_{n}\rvert^{\frac{N}{N-1}})-1\right]\lvert \omega_{n}\rvert^{q}\;\mathrm{d}x\rightarrow\int\limits_{W_{R}}\left[\exp(\lambda\lvert \omega\rvert^{\frac{N}{N-1}})-1\right]\lvert \omega\rvert^{q}\;\mathrm{d}x.
$$
Meanwhile, by Theorem \ref{tha} and $e^{t}\leq1+te^{t}$ for $t\geq0$,
\begin{align}
    \int\limits_{\R^{N}\backslash W_{R}}\left[\exp(\lambda\lvert \omega_{n}\rvert^{\frac{N}{N-1}})-1\right]\lvert \omega_{n}\rvert^{q}\;\mathrm{d}x&\leq\lambda\int\limits_{\R^{N}\backslash W_{R}}\exp(\lambda\lvert \omega_{n}\rvert^{\frac{N}{N-1}})\lvert \omega_{n}\rvert^{q+\frac{N}{N-1}}\;\mathrm{d}x\notag\\
    &\leq C(N,q,\lambda)\epsilon^{\frac{N}{N-1}}.\notag
\end{align}
Thus,
\begin{align}
    \mathrm{ATMSC}(N,p,q,\lambda,\beta)\leftarrow&\int\limits_{\R^{N}}\exp(\lambda\lvert \omega_{n}\rvert^{\frac{N}{N-1}})\lvert \omega_{n}\rvert^{q}\;\mathrm{d}x\notag\\
    =&\int\limits_{\R^{N}}\left[\exp(\lambda\lvert \omega_{n}\rvert^{\frac{N}{N-1}})-1\right]\lvert \omega_{n}\rvert^{q}\;\mathrm{d}x+\|\omega_{n}\|_{q}^{q}\notag\\
    \leq&\int\limits_{\R^{N}}\left[\exp(\lambda\lvert \omega\rvert^{\frac{N}{N-1}})-1\right]\lvert \omega\rvert^{q}\;\mathrm{d}x+C(N,q,\lambda)\epsilon^{\frac{N}{N-1}}+1.\notag
\end{align}
Since $\epsilon>0$ is arbitrary, we obtain
$$
\mathrm{ATMSC}(N,p,q,\lambda,\beta)\leq\int\limits_{\R^{N}}\left[\exp(\lambda\lvert \omega\rvert^{\frac{N}{N-1}})-1\right]\lvert \omega\rvert^{q}\;\mathrm{d}x+1.
$$
Obviously, $\omega\not=0$ and
\begin{align}
    \mathrm{ATMSC}(N,p,q,\lambda,\beta)&\leq\frac{1}{\|\omega\|_{q}^{q}}\int\limits_{\R^{N}}\left[\exp(\lambda\lvert \omega\rvert^{\frac{N}{N-1}})-1\right]\lvert \omega\rvert^{q}\;\mathrm{d}x+1\notag\\
    &=\frac{1}{\|\omega\|_{q}^{q}}\int\limits_{\R^{N}}\exp(\lambda\lvert \omega\rvert^{\frac{N}{N-1}})\lvert \omega\rvert^{q}\;\mathrm{d}x.\notag
\end{align}
Therefore, $\|\omega\|_{q}=1$ and
$$
\mathrm{ATMSC}(N,p,q,\lambda,\beta)=\int\limits_{\R^{N}}\exp(\lambda\lvert \omega\rvert^{\frac{N}{N-1}})\lvert \omega\rvert^{q}\;\mathrm{d}x.
$$
Moreover, $\omega$ is a maximizer of $\mathrm{ATMSC}(N,p,q,\lambda,\beta)$ with $\|F(\nabla\omega)\|_{N}\leq1$.

{\bfseries Case 3:} $0<\beta<N$.

Similarly as in the first case, we could show that for any $\epsilon>0$, there exists $R>0$ sufficiently large such that $\lvert\omega_{n}(x)\rvert\leq\epsilon$ when $F^{o}(x)\geq R$ and
$$
\int\limits_{W_{R}}\frac{\exp(\lambda(1-\frac{\beta}{N})\lvert \omega_{n}\rvert^{\frac{N}{N-1}})\lvert \omega_{n}\rvert^{p}}{F^{o}(x)^{\beta}}\;\mathrm{d}x\rightarrow\int\limits_{W_{R}}\frac{\exp(\lambda(1-\frac{\beta}{N})\lvert \omega\rvert^{\frac{N}{N-1}})\lvert \omega\rvert^{p}}{F^{o}(x)^{\beta}}\;\mathrm{d}x.
$$
Meanwhile, for $\delta=\delta(N,p,q,\beta)>0$ sufficiently small, by Theorem \ref{tha},
\begin{align}
    \int\limits_{\R^{N}\backslash W_{R}}\frac{\exp(\lambda(1-\frac{\beta}{N})\lvert \omega_{n}\rvert^{\frac{N}{N-1}})\lvert \omega_{n}\rvert^{p}}{F^{o}(x)^{\beta}}\;\mathrm{d}x&\leq\frac{1}{R^{\delta}}\int\limits_{\R^{N}\backslash W_{R}}\frac{\exp(\lambda(1-\frac{\beta}{N})\lvert \omega_{n}\rvert^{\frac{N}{N-1}})\lvert \omega_{n}\rvert^{p}}{F^{o}(x)^{\beta-\delta}}\;\mathrm{d}x\notag\\
    &\leq C(N,p,q,\lambda,\beta)\frac{1}{R^{\delta}}.\notag
\end{align}
Thus,
\begin{align}
    \mathrm{ATMSC}(N,p,q,\lambda,\beta)\leftarrow&\int\limits_{\R^{N}}\frac{\exp(\lambda(1-\frac{\beta}{N})\lvert \omega_{n}\rvert^{\frac{N}{N-1}})\lvert \omega_{n}\rvert^{p}}{F^{o}(x)^{\beta}}\;\mathrm{d}x\notag\\
    \leq&\int\limits_{\R^{N}}\frac{\exp(\lambda(1-\frac{\beta}{N})\lvert \omega\rvert^{\frac{N}{N-1}})\lvert \omega\rvert^{p}}{F^{o}(x)^{\beta}}\;\mathrm{d}x+C(N,p,q,\lambda,\beta)\frac{1}{R^{\delta}}.\notag
\end{align}
Since $R>0$ is sufficiently large, we obtain
$$
\mathrm{ATMSC}(N,p,q,\lambda,\beta)\leq\int\limits_{\R^{N}}\frac{\exp(\lambda(1-\frac{\beta}{N})\lvert \omega\rvert^{\frac{N}{N-1}})\lvert \omega\rvert^{p}}{F^{o}(x)^{\beta}}\;\mathrm{d}x.
$$
Obviously, $\omega\not=0$ and
\begin{align}
    \mathrm{ATMSC}(N,p,q,\lambda,\beta)&\leq\int\limits_{\R^{N}}\frac{\exp(\lambda(1-\frac{\beta}{N})\lvert \omega\rvert^{\frac{N}{N-1}})\lvert \omega\rvert^{p}}{F^{o}(x)^{\beta}}\;\mathrm{d}x\notag\\
    &\leq\frac{1}{\|\omega\|_{q}^{q(1-\frac{\beta}{N})}}\int\limits_{\R^{N}}\frac{\exp(\lambda(1-\frac{\beta}{N})\lvert \omega\rvert^{\frac{N}{N-1}})\lvert \omega\rvert^{p}}{F^{o}(x)^{\beta}}\;\mathrm{d}x.\notag
\end{align}
Therefore, $\|\omega\|_{q}=1$ and
$$
\mathrm{ATMSC}(N,p,q,\lambda,\beta)=\int\limits_{\R^{N}}\frac{\exp(\lambda(1-\frac{\beta}{N})\lvert \omega\rvert^{\frac{N}{N-1}})\lvert \omega\rvert^{p}}{F^{o}(x)^{\beta}}\;\mathrm{d}x.
$$
Moreover, $\omega$ is a maximizer of $\mathrm{ATMSC}(N,p,q,\lambda,\beta)$ with $\|F(\nabla\omega)\|_{N}\leq1$.

Finally, we prove that $\|F(\nabla\omega)\|_{N}=1$ when $\omega$ is a maximizer of $\mathrm{ATMSC}(N,p,q,\lambda,\beta)$. We assume $\|F(\nabla\omega)\|_{N}=\gamma^{-1}$ with $\gamma\geq1$, where we can easily exclude the case $\|F(\nabla u)\|_{N}=0$. Let
$$
v(x)=\gamma\omega(\gamma^{\frac{q}{N}}x).
$$
Obviously,
$$
\|F(\nabla v)\|_{N}=1,\;\|v\|_{q}=\|\omega\|_{q}
$$
and
\begin{align}
    &\frac{1}{\|v\|_{q}^{q(1-\frac{\beta}{N})}}\int\limits_{\R^{N}}\frac{\exp(\lambda(1-\frac{\beta}{N})\lvert v\rvert^{\frac{N}{N-1}})\lvert v\rvert^{p}}{F^{o}(x)^{\beta}}\;\mathrm{d}x\notag\\
    =&\gamma^{p-q(1-\frac{\beta}{N})}\frac{1}{\|\omega\|_{q}^{q(1-\frac{\beta}{N})}}\int\limits_{\R^{N}}\frac{\exp(\gamma^{\frac{N}{N-1}}\lambda(1-\frac{\beta}{N})\lvert \omega\rvert^{\frac{N}{N-1}})\lvert \omega\rvert^{p}}{F^{o}(x)^{\beta}}\;\mathrm{d}x\notag\\
    \geq&\frac{1}{\|\omega\|_{q}^{q(1-\frac{\beta}{N})}}\int\limits_{\R^{N}}\frac{\exp(\lambda(1-\frac{\beta}{N})\lvert \omega\rvert^{\frac{N}{N-1}})\lvert \omega\rvert^{p}}{F^{o}(x)^{\beta}}\;\mathrm{d}x\notag\\
    =&\mathrm{ATMSC}(N,p,q,\lambda,\beta)\notag.
\end{align}
Therefore, $\gamma=1$.

Maximizers for $\mathrm{ATMSC}(N,q,\lambda,\beta)$ can be proved similarly.
\end{proof}
\subsection{Proof of Theorem \ref{th1.2}}
\begin{proof}
Let $u$ be a maximizer of
$$
\mathrm{ATMSC}(N,p,q,\lambda,\beta)=\sup\limits_{\|F(\nabla u)\|_{N}\leq 1}\frac{1}{\|u\|_{q}^{q(1-\frac{\beta}{N})}}\int\limits_{\R^{N}}\frac{\exp(\lambda(1-\frac{\beta}{N})\lvert u\rvert^{\frac{N}{N-1}})\lvert u\rvert^{p}}{F^{o}(x)^{\beta}}\;\mathrm{d}x.
$$
Set $v=\lvert u\rvert^{\star}$, by Lemma \ref{le2.2} and Lemma \ref{le2.3}, we have $\|F(\nabla v)\|_{N}\leq1$, $\|u\|_{q}=\|v\|_{q}$ and
$$
\int\limits_{\R^{N}}\frac{\exp(\lambda(1-\frac{\beta}{N})\lvert u\rvert^{\frac{N}{N-1}})\lvert u\rvert^{p}}{F^{o}(x)^{\beta}}\;\mathrm{d}x\leq\int\limits_{\R^{N}}\frac{\exp(\lambda(1-\frac{\beta}{N})v^{\frac{N}{N-1}})v^{p}}{F^{o}(x)^{\beta}}\;\mathrm{d}x.
$$
By the fact that $u$ is a maximizer of $\mathrm{ATMSC}(N,p,q,\lambda,\beta)$, we conclude
$$
\int\limits_{\R^{N}}\frac{\exp(\lambda(1-\frac{\beta}{N})\lvert u\rvert^{\frac{N}{N-1}})\lvert u\rvert^{p}}{F^{o}(x)^{\beta}}\;\mathrm{d}x=\int\limits_{\R^{N}}\frac{\exp(\lambda(1-\frac{\beta}{N})v^{\frac{N}{N-1}})v^{p}}{F^{o}(x)^{\beta}}\;\mathrm{d}x.
$$
Then by the condition for equality to hold in Lemma \ref{le2.3},
$$
\exp(\lambda(1-\frac{\beta}{N})\lvert u\rvert^{\frac{N}{N-1}})\lvert u\rvert^{p}=\exp(\lambda(1-\frac{\beta}{N})v^{\frac{N}{N-1}})v^{p},
$$
which means $\lvert u\rvert=v=\lvert u\rvert^{\star}$.

Maximizers for $\mathrm{ATMSC}(N,q,\lambda,\beta)$ can be proved similarly.
\end{proof}
\section{Proof of Theorem \ref{th1.3}}\label{section4}

For every $a>0$, $b>0$, $q>1$ and $0\leq\beta<N$, let
$$
f(\lambda):=\mathrm{ATMSC}(N,q,\lambda,\beta)
$$
and
$$
g_{a,b}(\lambda):=\Lambda_{a,b}(N,q,\lambda,\beta).
$$
By \eqref{AA},
$$
g_{a,b}(\lambda)=\sup\limits_{t\in(0,\lambda)}\left(\frac{1-(\frac{t}{\lambda})^{a\frac{N-1}{N}}}{(\frac{t}{\lambda})^{b\frac{N-1}{N}}}\right)^{\frac{q}{b}(1-\frac{\beta}{N})}f(t).
$$

\begin{lemma}\label{le4.1}
$f$ is continuous on $(0,\lambda_{N})$.
\end{lemma}
\begin{proof}
Let $\varepsilon_{n}\rightarrow0^{+}$ and $u_{n}\in D^{N,q}(\R^{N})$ be Wulff symmetric with
$$
\|F(\nabla u_{n})\|_{N}=1,\;\|u_{n}\|_{q}=1,
$$
$$
f(\lambda+\varepsilon_{n})=\int\limits_{\R^{N}}\frac{\Phi_{N,q,\beta}((\lambda+\varepsilon_{n})(1-\frac{\beta}{N})\lvert u_{n}\rvert^{\frac{N}{N-1}})}{F^{o}(x)^{\beta}}\;\mathrm{d}x.
$$
Notice that
$$
0\leq f(\lambda+\varepsilon_{n})-f(\lambda)\leq\int\limits_{\R^{N}}\frac{\Phi_{N,q,\beta}((\lambda+\varepsilon_{n})(1-\frac{\beta}{N})\lvert u_{n}\rvert^{\frac{N}{N-1}})-\Phi_{N,q,\beta}(\lambda(1-\frac{\beta}{N})\lvert u_{n}\rvert^{\frac{N}{N-1}})}{F^{o}(x)^{\beta}}\;\mathrm{d}x
$$
and for any $x\not=0$,
\begin{align}
    \|u_{n}\|_{q}^{q}\geq\int_{0}^{F^{o}(x)}u_{n}(r)^{q}\int_{\partial W_{r}}\frac{1}{\lvert\nabla F^{o}(x)\rvert}\;\mathrm{d}\sigma\mathrm{d}r\geq\kappa_{N}F^{o}(x)^{N}u_{n}(F^{o}(x))^{q},\notag
\end{align}
which implys there exists $R>0$ sufficiently large such that $\lvert u_{n}(x)\rvert\leq1$ when $F^{o}(x)\geq R$. Then by Theorem \ref{tha},
\begin{align}
    &\int\limits_{\R^{N}\backslash W_{R}}\frac{\Phi_{N,q,\beta}((\lambda+\varepsilon_{n})(1-\frac{\beta}{N})\lvert u_{n}\rvert^{\frac{N}{N-1}})-\Phi_{N,q,\beta}(\lambda(1-\frac{\beta}{N})\lvert u_{n}\rvert^{\frac{N}{N-1}})}{F^{o}(x)^{\beta}}\;\mathrm{d}x\notag\\
    \leq&\left[\Phi_{N,q,\beta}((\lambda+\varepsilon_{n})(1-\frac{\beta}{N}))-\Phi_{N,q,\beta}(\lambda(1-\frac{\beta}{N}))\right]\int\limits_{\R^{N}\backslash W_{R}}\frac{\lvert u_{n}\rvert^{j_{N,q,\beta}\frac{N}{N-1}}}{F^{o}(x)^{\beta}}\;\mathrm{d}x\rightarrow0,\notag
\end{align}
where
$$
j_{N,q,\beta}=\left\{\begin{array}{cc}
\left\lfloor\frac{q(N-1)}{N}(1-\frac{\beta}{N})\right\rfloor+1 & {\mathrm{if}}\;\beta>0 \\
\left\lceil\frac{q(N-1)}{N}\right\rceil & {\mathrm{if}}\;\beta=0
\end{array}\right..
$$
On the other hand,
$$
\frac{\Phi_{N,q,\beta}((\lambda+\varepsilon_{n})(1-\frac{\beta}{N})\lvert u_{n}\rvert^{\frac{N}{N-1}})}{F^{o}(x)^{\beta}}\rightarrow\frac{\Phi_{N,q,\beta}(\lambda(1-\frac{\beta}{N})\lvert u_{n}\rvert^{\frac{N}{N-1}})}{F^{o}(x)^{\beta}}\;\mathrm{a.e.}\;\mathrm{in}\;W_{R}
$$
and by Lemma 2.1 in \cite{Yang-2012} and Theorem \ref{tha}, there exists $\theta>0$ sufficiently small such that $\frac{\Phi_{N,q,\beta}((\lambda+\varepsilon_{n})(1-\frac{\beta}{N})\lvert u_{n}\rvert^{\frac{N}{N-1}})}{F^{o}(x)^{\beta}}$ is uniformly bounded in $L^{1+\theta}(W_{R})$. Thus, by Lemma \ref{le2.4},
$$
\int\limits_{W_{R}}\frac{\Phi_{N,q,\beta}((\lambda+\varepsilon_{n})(1-\frac{\beta}{N})\lvert u_{n}\rvert^{\frac{N}{N-1}})}{F^{o}(x)^{\beta}}\;\mathrm{d}x\rightarrow\int\limits_{W_{R}}\frac{\Phi_{N,q,\beta}(\lambda(1-\frac{\beta}{N})\lvert u_{n}\rvert^{\frac{N}{N-1}})}{F^{o}(x)^{\beta}}\;\mathrm{d}x.
$$
Hence, $f(\lambda+\varepsilon_{n})-f(\lambda)\rightarrow0^{+}$ and we can obtain $f(\lambda-\varepsilon_{n})-f(\lambda)\rightarrow0^{-}$ similarly. Therefore, $f$ is continuous on $(0,\lambda_{N})$.
\end{proof}

Using the lemma, we give the proof of Theorem \ref{th1.3}.
\begin{proof}
It is enough to show
\begin{align}\label{eq4.1}
    \varlimsup\limits_{t\rightarrow0}\left(\frac{1-(\frac{t}{\lambda})^{a\frac{N-1}{N}}}{(\frac{t}{\lambda})^{b\frac{N-1}{N}}}\right)^{\frac{q}{b}(1-\frac{\beta}{N})}f(t)<g_{a,b}(\lambda)
\end{align}
and
\begin{align}\label{eq4.2}
    \varlimsup\limits_{t\rightarrow\lambda}\left(\frac{1-(\frac{t}{\lambda})^{a\frac{N-1}{N}}}{(\frac{t}{\lambda})^{b\frac{N-1}{N}}}\right)^{\frac{q}{b}(1-\frac{\beta}{N})}f(t)<g_{a,b}(\lambda).
\end{align}
Indeed, if \eqref{eq4.1} and \eqref{eq4.2} are true, then there is $t_{\lambda}\in(0,\lambda)$ such that
$$
\left(\frac{1-(\frac{t_{\lambda}}{\lambda})^{a\frac{N-1}{N}}}{(\frac{t_{\lambda}}{\lambda})^{b\frac{N-1}{N}}}\right)^{\frac{q}{b}(1-\frac{\beta}{N})}f(t_{\lambda})=g_{a,b}(\lambda).
$$
Let $u_{\lambda}$ be a maximizer of $\mathrm{ATMSC}(N,q,t_{\lambda},\beta)$ and set
$$
v_{\lambda}(x):=\left(\frac{t_{\lambda}}{\lambda}\right)^{\frac{N-1}{N}}u_{\lambda}(\gamma x),
$$
where
$$
\gamma=\left(\frac{(\frac{t_{\lambda}}{\lambda})^{b\frac{N-1}{N}}\|u_{\lambda}\|_{q}^{b}}{1-(\frac{t_{\lambda}}{\lambda})^{a\frac{N-1}{N}}}\right)^{\frac{q}{Nb}}.
$$
Then $v_{\lambda}$ is a maximizer of $\Lambda_{a,b}(N,q,\lambda,\beta)$. On the other hand, let $v_{\lambda}$ be a maximizer of $\Lambda_{a,b}(N,q,\lambda,\beta)$ and set
$$
u_{\lambda}(x):=\frac{v_{\lambda}(x)}{\|F(\nabla v_{\lambda})\|_{N}}.
$$
Then $u_{\lambda}$ is a maximizer of $\mathrm{ATMSC}(N,q,\lambda\|F(\nabla v_{\lambda})\|_{N}^{\frac{N}{N-1}},\beta)$.

Now, we show \eqref{eq4.1} and \eqref{eq4.2}. Firstly, since
$$
\varlimsup\limits_{t\rightarrow\lambda}\left(\frac{1-(\frac{t}{\lambda})^{a\frac{N-1}{N}}}{(\frac{t}{\lambda})^{b\frac{N-1}{N}}}\right)^{\frac{q}{b}(1-\frac{\beta}{N})}f(t)=0,
$$
\eqref{eq4.2} is obvious. Thus, we only need to study
$$
\varlimsup\limits_{t\rightarrow0}\left(\frac{1-(\frac{t}{\lambda})^{a\frac{N-1}{N}}}{(\frac{t}{\lambda})^{b\frac{N-1}{N}}}\right)^{\frac{q}{b}(1-\frac{\beta}{N})}f(t).
$$
Let $t_{n}\rightarrow0^{+}$ and $u_{n}\in D^{N,q}(\R^{N})$ be Wulff symmetric with
$$
\|F(\nabla u_{n})\|_{N}=1,\;\|u_{n}\|_{q}=1,
$$
$$
f(t_{n})=\int\limits_{\R^{N}}\frac{\Phi_{N,q,\beta}(t_{n}(1-\frac{\beta}{N})\lvert u_{n}\rvert^{\frac{N}{N-1}})}{F^{o}(x)^{\beta}}\;\mathrm{d}x.
$$

{\bfseries Case 1:} $\beta>0$.

Notice that
\begin{align}
    f(t_{n})&=\sum\limits_{j=j_{N,q,\beta}}^{+\infty}\int\limits_{\R^{N}}\frac{t_{n}^{j}(1-\frac{\beta}{N})^{j}\lvert u_{n}\rvert^{\frac{N}{N-1}j}}{j!F^{o}(x)^{\beta}}\;\mathrm{d}x\notag\\
    &=\int\limits_{\R^{N}}\frac{t_{n}^{j_{N,q,\beta}}(1-\frac{\beta}{N})^{j_{N,q,\beta}}\lvert u_{n}\rvert^{\frac{N}{N-1}j_{N,q,\beta}}}{j_{N,q,\beta}!F^{o}(x)^{\beta}}\;\mathrm{d}x+\sum\limits_{j=j_{N,q,\beta}+1}^{+\infty}\int\limits_{\R^{N}}\frac{t_{n}^{j}(1-\frac{\beta}{N})^{j}\lvert u_{n}\rvert^{\frac{N}{N-1}j}}{j!F^{o}(x)^{\beta}}\;\mathrm{d}x\notag\\
    &\leq\frac{t_{n}^{j_{N,q,\beta}}}{j_{N,q,\beta}!}\int\limits_{\R^{N}}\frac{\lvert u_{n}\rvert^{\frac{N}{N-1}j_{N,q,\beta}}}{F^{o}(x)^{\beta}}\;\mathrm{d}x+t_{n}^{j_{N,q,\beta}+1}\int\limits_{\R^{N}}\frac{\exp(t_{n}\lvert u_{n}\rvert^{\frac{N}{N-1}})\lvert u_{n}\rvert^{\frac{N}{N-1}(j_{N,q,\beta}+1)}}{F^{o}(x)^{\beta}}\;\mathrm{d}x.\notag
\end{align}
Then by Theorem \ref{tha},
\begin{align}
    \left(\frac{1-(\frac{t_{n}}{\lambda})^{a\frac{N-1}{N}}}{(\frac{t_{n}}{\lambda})^{b\frac{N-1}{N}}}\right)^{\frac{q}{b}(1-\frac{\beta}{N})}f(t_{n})\rightarrow0.\notag
\end{align}

{\bfseries Case 2:} $\beta=0$ and $\frac{q(N-1)}{N}\not\in\N$.

Similarly as in the first case.

{\bfseries Case 3:} $\beta=0$ and $\frac{q(N-1)}{N}\in\N$.

By Theorem \ref{tha} and Lemma \ref{le2.4},
$$
\int\limits_{\R^{N}}\exp(t_{n}\lvert u_{n}\rvert^{\frac{N}{N-1}})\lvert u_{n}\rvert^{q}\;\mathrm{d}x\rightarrow1.
$$
Then
\begin{align}
    &\left(\frac{1-(\frac{t_{n}}{\lambda})^{a\frac{N-1}{N}}}{(\frac{t_{n}}{\lambda})^{b\frac{N-1}{N}}}\right)^{\frac{q}{b}}f(t_{n})\notag\\
    \leq&\left(\frac{1-(\frac{t_{n}}{\lambda})^{a\frac{N-1}{N}}}{(\frac{t_{n}}{\lambda})^{b\frac{N-1}{N}}}\right)^{\frac{q}{b}}\frac{t_{n}^{\frac{q(N-1)}{N}}}{\left(\frac{q(N-1)}{N}\right)!}\int\limits_{\R^{N}}\exp(t_{n}\lvert u_{n}\rvert^{\frac{N}{N-1}})\lvert u_{n}\rvert^{q}\;\mathrm{d}x\notag\\
    \leq&\frac{\lambda^{\frac{q(N-1)}{N}}}{\left(\frac{q(N-1)}{N}\right)!}.\notag
\end{align}
Thus,
$$
\varlimsup\limits_{t\rightarrow0}\left(\frac{1-(\frac{t}{\lambda})^{a\frac{N-1}{N}}}{(\frac{t}{\lambda})^{b\frac{N-1}{N}}}\right)^{\frac{q}{b}(1-\frac{\beta}{N})}f(t)\leq\frac{\lambda^{\frac{q(N-1)}{N}}}{\left(\frac{q(N-1)}{N}\right)!}.
$$
Hence, \eqref{eq4.1} holds when
$$
g_{a,b}(\lambda)>\frac{\lambda^{\frac{q(N-1)}{N}}}{\left(\frac{q(N-1)}{N}\right)!}.
$$
Therefore, $\Lambda_{a,b}(N,q,\lambda, 0)$ is not attained only if
$$
g_{a,b}(\lambda)\leq\frac{\lambda^{\frac{q(N-1)}{N}}}{\left(\frac{q(N-1)}{N}\right)!}.
$$
\end{proof}

\noindent{\bf Acknowledgements}  This work was discussed and completed during the second author's visit to Central China Normal University. We would like to express our gratitude to Central China Normal University for providing a conductive learning atmosphere and environment.\\

\def\refname{References }

\end{document}